\documentclass[reqno, 12pt]{amsart}
\usepackage{amsfonts,epsfig}
\usepackage{latexsym}
\usepackage{amssymb}
\usepackage{amsmath}
\usepackage{amsthm}
\usepackage{graphics}
\usepackage{hyperref}
\def\C{{\mathbb C}}

\def\Z{{\mathbb Z}}
\def\Zp{{\mathbb Z}_p}

\def\ord{\operatorname{ord}}

\def\Fq{{{\mathbb F}_q}}

\newtheorem{cor}{Corollary}

\newtheorem{thm}{Theorem}

\theoremstyle{definition}
\newtheorem{defn}{Definition}

\theoremstyle{remark}
\newtheorem{rem}{Remark}        
\newtheorem{rems}{Remarks}      
\newtheorem{example}{Example}
\newtheorem{examples}{Examples}

\begin{document}
\title[A construction of ${\mathfrak v}$-adic modular forms]
{A construction of ${\mathfrak v}$-adic modular forms}
\author{David Goss}
\address{Department of Mathematics\\The Ohio State University\\231
W.\ $18^{\rm th}$ Ave.\\Columbus, Ohio 43210}

\email{dmgoss@gmail.com}

\date{June 16, 2013}

\begin{abstract}
The classical theory of $p$-adic (elliptic) modular forms arose in the 1970's from the work of
 J.-P.\ Serre \cite{se1} who took $p$-adic limits of the $q$-expansions of these forms. It was soon expanded
by N.\ Katz \cite{ka1} with a more functorial approach.  Since then the theory has grown
in a variety of directions. In the late 1970's, the theory of modular forms
associated to Drinfeld modules was born in analogy with elliptic modular forms
\cite{go1}, \cite{go2}. 
The associated expansions at $\infty$ are quite complicated and no obvious limits at
finite primes ${\mathfrak v}$
were apparent. Recently, however, there has been progress in the $\mathfrak v$-adic 
theory, \cite{vi1}.
Also recently, A.\ Petrov \cite{pe1}, building on previous work of \cite{lo1}, showed  that there
is an intermediate expansion at $\infty$ called the ``$A$-expansion,'' and he
constructed families of cusp forms with such expansions. It is our purpose in this
note to show that Petrov's results also lead to interesting ${\mathfrak v}$-adic cusp
forms \`a la Serre. 
Moreover the existence of these forms allows us to readily conclude a mysterious decomposition
of the associated Hecke action.
\end{abstract}

\maketitle

\section{Introduction}\label{intro} 
Let ${\mathcal G}_{2k}(z)$ be the classical elliptic modular Eisenstein series with $q$-expansion
$\sum c_nq^n$. It is well-known that the non-constant coefficients are of the form
\begin{equation}\label{intro1}
\sum_{d\mid n}d^{2k-1}\,.
\end{equation}
Let $p$ be a fixed prime with {\it weight space} ${\mathbb S}_p=\varprojlim_j \Z/(p-1)p^j$. 
Let $s_p\in {\mathbb S}_p$ and let $k_i$ be a collection of positive integers converging to
$s_p$. If $d\mid n$ is prime to $p$, then $d^{2k_i-1}\to d^{2s_p-1}$. On the other hand, 
if $p\mid d$ then the above powers of $d$ converge to $0$ $p$-adically. Thus the Eisenstein
series very strongly suggest the existence of a good $p$-adic theory of modular forms.

In \cite{pe1}, A.\ Petrov gives a powerful construction of {\it cuspidal} eigenforms which looks remarkably
similar to the construction of Eisenstein series (see Remark \ref{Gexpn12} below). We exploit
this analogy here to construct examples of $\mathfrak v$-adic modular forms, for 
finite primes $\mathfrak v$ of $\Fq[\theta]$. We further show how $\mathfrak v$-adic continuity
allows one to decompose the Hecke action on Petrov's forms.

It is my pleasure to thank Aleks Petrov and Federico Pellarin for their
kind advice and help with this paper.

\section{Review of the results of L\'opez and Petrov}\label{review}
In this section we briefly review the basic set-up of modular forms in finite characteristic and the
results of B.\ L\'opez \cite{lo1}, and A.\ Petrov \cite{pe1}, whose exposition we follow 
(outside of a few small
changes of notation which will be made clear below).

Let $q=p^{m_0}$ and set $A:=\Fq[\theta]$ (N.B.: in \cite{pe1}, the author  uses ``$T$'' instead of ``$\theta$'').
Put $K:=\Fq(\theta)$ and $K_\infty:=\Fq((1/\theta))$. So $K_\infty$ is complete with respect to the
absolute value $\vert ?\vert_\infty$ at $\infty$, and we let $\C_\infty$ be the completion of
a fixed algebraic closure $\bar K_\infty$ of $K_\infty$ equipped with the canonical 
extension of $\vert ?\vert_\infty$.

Let $\Omega$ be the Drinfeld upper half-plane; $\Omega$ is a connected $1$-dimensional rigid
analytic space. The set of $\C_\infty$ points of $\Omega$ is $\C_\infty \setminus K_\infty$. Let $z$ be
one such geometric point. Then, as in \cite{go1}, we set
\begin{equation}\label{imdist}
\vert z\vert_i:=\inf_{x\in K}\{\vert z-x\vert_\infty\}\,.
\end{equation}
If $\nu$ is in the value group of $\C_\infty$, then the set $\Omega_\nu:=
\{z\in \Omega\mid \vert z\vert_i\geq
\nu\}$ is an admissible open subset of $\Omega$. 

Let $\Gamma={\rm GL}_2(A)$ and let $k$ be a positive integer. Let $m$ be an integer with
$0\leq m<q-1$. In analogy with the classical 
${\rm SL}_2(\Z)$-theory, one can readily define the notion of modular forms of weight $k$ and
type $m$ associated to $\Gamma$ (see \cite{go1}, \cite{go2}, \cite{ge1}). 
The space of such forms is denoted
$M_{k,m}=M_{k,m}(\Gamma)$ and is finite dimensional. The subspaces of cusp forms and double-cusp forms
are denoted $S_{k,m}=S_{k,m}(\Gamma)$ and $S^2_{k,m}=S^2_{k,m}(\Gamma)$.

Let $C$ be the Carlitz module and $\tilde \pi$ its period;  the lattice $\Lambda_C$ associated 
to  $C$ is then $A\tilde \pi$. Let $e_C(z)$ be the exponential of $C$ and put 
\begin{equation}\label{u1}
u:=u(z)=1/e_C({\tilde \pi}z)\,.
\end{equation}
 (N.B.: in \cite{pe1}, the author uses $t(z)$ where we have used $u(z)$.) The function $u$ is regular on $\Omega$.
For $a\in A_+$ of degree $d$, we set
\begin{equation}\label{u2}
u_a:=u(az)\,.
\end{equation}
It is very easy to see that $u_a=u^{q^d}+\{{\rm higher\, terms\, in}\, u\}$.

Let $f(z)\in M_{k,m}$. Then $f$ has a unique expansion $f(z)=\sum_{i\geq 0} a_i u(z)^i$
converging to $f$ for $z\in \Omega_\nu$ with $\nu$ sufficiently large. 
In fact, as the space $\Omega$ is 
rigid-analytically connected, the above $u$-expansion uniquely determines the form $f(z)$.

Let $\Lambda\subset \C_\infty$ be an arbitrary $\Fq$-lattice; that is $\Lambda$ is a discrete (finite intersection
with each ball around the origin), $\Fq$ submodule of $\C_\infty$. For instance, $\Lambda$ could be
$\Lambda_C$ or the $g$-torsion points, $g\in A$, inside $\C_\infty$ of the Carlitz module. To such a lattice 
$\Lambda$ one associates the exponential function 
\begin{equation}\label{expfunc}
e_\Lambda (z)=z\prod_{0\neq \lambda\in \Lambda}(1-z/\lambda)\,,
\end{equation}
which is easily seen to be an entire $\Fq$-linear function in $z$. As such, the derivative in $z$ 
of $e_\Lambda(z)$ is identically $1$.  

We set
\begin{equation}\label{tfunc}
t_\Lambda(z):=e_\Lambda^\prime(s)/e_\Lambda(z)=1/e_\Lambda(z)=\sum_{\lambda\in 
\Lambda}1/(z+\lambda)\,,
\end{equation}
and, for a positive integer $k$,
\begin{equation}\label{sfunc}
S_{k,\Lambda}(z):=\sum_{\lambda\in \Lambda} 1/(z+\lambda)^k\,.
\end{equation}
As in \S 3 of \cite{ge1}, there is a monic polynomial, $G_{k,\Lambda}(X)$, of degree $k$, such that
$$G_{k,\Lambda}(t_\Lambda(z))=S_{k,\Lambda}(z)\,.$$

\begin{defn}\label{gn}
We set $G_n(z):=G_{n,\Lambda_C}(z)$.
\end{defn}
It is easy to see that $G_n(z)$ has coefficients in $K$.

\begin{defn}\label{Gexpn}
Let $f(z)\in M_{k,m}$ as above. We say that $f(z)$ has an {\it $A$-expansion} if there exists a positive
integer $n$ and coefficients $c_0,\, c_a,$ $a\in A_+$, such that 
\begin{equation}\label{Gexpn1}
f(z)=c_0(f)+\sum_{a\in A_+} c_a G_n(u_a)\,.
\end{equation}\end{defn}

\begin{rems}\label{Gexpn2}
a. If such an expansion exists for a given $n$, then, the argument of L\'opez \cite{lo1} shows that it is unique.\\
b If $f$ is a Hecke eigenform, then Petrov \cite{pe1} establishes that $n$ is uniquely determined by $f$.\\
c.  It is suspected, but not yet known, that all such expansions (\ref{Gexpn1}) are, in fact, uniquely
determined by $f$ for arbitrary forms $f$.\\
d. The form of the polynomials $G_n(X)$ we use is the correct one and the reader should simply ignore
the normalization used in \cite{pe1}.
\end{rems}

\begin{examples}\label{Gexpn3}
 The first class of examples of such expansions are the 
Eisenstein series: Let $k\equiv 0\pmod{q-1}$ be a
positive integer and set
\begin{equation}\label{Gexpn4}
{\mathcal G}_k(z):=\sum (az+b)^{-k}
\end{equation}
where, as usual, we sum over all nonzero $(a,b)\in A\times A$. As in \cite{ge1}, we see
\begin{equation}\label{Gexpn5}
-{\mathcal G}_k(z)/\tilde{\pi}^k=\sum_{a\in A_+} (\tilde \pi a)^{-k}+\sum_{a\in A_+} G_k(u_a)\,.
\end{equation}
\end{examples}

The second class of examples involve cusp forms. There are two fundamental cusp forms in the theory; 
$\Delta$ and $h$. Here $\Delta$ is defined in 
analogy with its classical counterpart and $h$ is a certain $q-1$-st root of $\Delta$. 
In \cite{lo1} L\'opez establishes the following $A$-expansions for these forms, thereby establishing that
$A$-expansions {\it also} work for cusp forms.

\begin{examples}\label{Gexpn6}
We have
\begin{equation}\label{Gexpn7}
\Delta=\sum_{a\in A_+} a^{q(q-1)}G_{q-1}(u_a)=\sum_{a\in A_+} a^{q(q-1)}u_a^{q-1}\,,\end{equation}
and
\begin{equation}\label{Gexpn8}
h=\sum_{a\in A_+} a^qG_1(u_a)=\sum_{a\in A_+} a^qu_a\,.\end{equation}
\end{examples}

In \cite{pe1}, Petrov builds on the above examples and constructs families of cusp forms 
with $A$-expansions which we now recall.
Let $\ord_p(j)$ be the $p$-adic valuation of an integer $j$.

\begin{thm}\label{Gexpn9}
Let $k,n$ be two positive integers such that $k-2n$ is a positive multiple of $q-1$ and 
$n\leq p^{\ord_p(k-n)}$. Then 
\begin{equation}\label{Gexpn10}
f_{k,n}:=\sum_{a\in A_+} a^{k-n} G_n(u_a)
\end{equation}
is an element of $S_{k,m}$ with $n\equiv m\pmod{q-1}$.\end{thm}
It is easy to see that  the $u$-expansion of $f_{k,n}$ has coefficients in $K$.

\begin{rem}\label{Gexpn11}
As Petrov points out in Remark 1.4 of \cite{pe1}, the condition $n\leq p^{\ord_p(k-n)}$ is equivalent to having the $p$-adic
expansion of $k$ and $n$ agree up to the $\lfloor \log_p(n)\rfloor$-th digit.
\end{rem}

Petrov's construction has a remarkable analog classically, but in the realm of Eisenstein
series, {\it not} cusp forms, as the following remark makes clear.

\begin{rem}\label{Gexpn12}
Let ${\mathcal G}_{2k} (z):=\sum_{(0,0)\neq (m,n)} (mz+n)^{-2k}$ be classical Eisenstein series and
set ${\mathcal E}_{2k}(z):={\mathcal G}_{2k}(z)/2\zeta(2k)$ where $\zeta(z)$ is the Riemann zeta
function. It is known that ${\mathcal E}_{2k}$ has the {\it Lambert expansion} 
\begin{equation}\label{Gexpn13}
{\mathcal E}_{2k}(z)=1+\frac{2}{\zeta(1-2k)}\sum_{n=1}^\infty n^{2k-1}\frac{q^n}{1-q^n}\,.
\end{equation}
If we put $q_n:=q^n$ and $G(x):=x/(1-x)$, we can rewrite Equation \ref{Gexpn13}
as 
\begin{equation}\label{Gexpn14}
{\mathcal E}_{2k}(z)=1+\frac{2}{\zeta(1-2k)}\sum_{n=1}^\infty n^{2k-1} G(q_n)\,,
\end{equation}
and the analogy with Equation \ref{Gexpn10} is clear. Thus, in a sense, Petrov's construction is both
analogous to, and  ``orthogonal to'' (i.e., lying in the space of cusp forms), 
the classical construction of Eisenstein series.
\end{rem}

\section{${\mathfrak v}$-adic modular forms in the sense of Serre}\label{vserre}
In this section we interpolate the cusp forms of Theorem \ref{Gexpn9} at finite primes of $A$. 
Let ${\mathfrak v}\in {\rm Spec}(A)$ be a fixed finite prime of degree $d$ with 
completions $A_{\mathfrak v},\ K_{\mathfrak v}$
respectively of $A$ and $K$. Let $p_\mathfrak v$ be the monic generator of $\mathfrak v$,  and let
$A_{{\mathfrak v},+}\subset A_{\mathfrak v}^\ast$ be those monic elements {\it also} prime to ${\mathfrak v}$. 

\begin{defn}\label{vad1}
We define the {\it ${\mathfrak v}$-adic weight space} $\mathbb S_{\mathfrak v}$ by
\begin{equation}\label{vad2}
{\mathbb S}_{\mathfrak v}:=\varprojlim_t \Z/((q^d-1)p^t)=\Z/(q^d-1)\times \Zp\,.
\end{equation}\end{defn}

If $a\in A_{{\mathfrak v},+}$ then put $a=a_0a_1$ where $a_0\in A_v^\ast$ is the $q^d-1$-st
root of unity with $a_0\equiv a \bmod \mathfrak v$ and $a_1\equiv 1 \bmod \mathfrak v$.
Thus, if $s_{\mathfrak v}=(x_{\mathfrak v},y_{\mathfrak v})\in {\mathbb S}_{\mathfrak v}$, 
then one defines $a^{s_{\mathfrak v}}:=a_0^{x_\mathfrak v}a_1^{y_\mathfrak v}$ in complete
analogy with classical theory; the function $s_{\mathfrak v}\mapsto a^{s_{\mathfrak v}}$ is
readily seen to be continuous from ${\mathbb S}_{\mathfrak v}$ to $A_{\mathfrak v}^\ast$. 

Let $s_{\mathfrak v}\in {\mathbb S}_\mathfrak v$.

\begin{defn}\label{vad2.1}
We set
\begin{equation}\label{vad2.2}
\hat{f}_{s_\mathfrak v,n}:=\sum_{a\in A_{{\mathfrak v},+}} a^{s_\mathfrak v} G_n (u_a)\,,
\end{equation} which is readily seen to converge to an element in $A_{\mathfrak v}[[u]]\otimes K$.\end{defn}

Let $f(u)=\sum a_nu^n\in A_{\mathfrak v}[[u]]\otimes K$ be an arbitrary power series. We set
\begin{equation}\label{vad3}
\ord_{\mathfrak v}(f):=\inf_n \{\ord_{\mathfrak v}(a_n)\}\,.
\end{equation}

\begin{defn}\label{vad4}
We say that a power series $f(u)\in A_{\mathfrak v}[[u]]\otimes K$ is a {\it ${\mathfrak v}$-adic modular form in the sense of Serre}
if there exists a sequence $f_i\in M_{k_i,m}$ such $\ord_{\mathfrak v}(f-f_i)$ tends to $\infty$ with $i$.
\end{defn}

In other words, the power series $f(u)$ is a ${\mathfrak v}$-adic limit of the $u$-expansions of true modular forms.

Now let $f_{k,n}$ be as constructed in Theorem \ref{Gexpn9} and put $\alpha:=k-n$. As $k-2n$ is
a positive multiple of $q-1$, we find that $\alpha\equiv n\pmod{q-1}$. Moreover,
by Remark \ref{Gexpn12} we see that $\alpha$ is divisible by a high power of $p$. With this in mind,
we write $n$ $q^d$-adically as
\begin{equation}\label{vad5}
n=\sum_{e=0}^t n_eq^{de}\quad n_t\neq 0\,.
\end{equation}

\begin{defn}\label{vad6}
We let ${\mathbb S}_{\mathfrak v}(n)\subseteq \mathbb S_{\mathfrak v}$ be the open subset consisting of those $s_{\mathfrak v}=(x_{\mathfrak v},y_{\mathfrak v})$
such that $x_{\mathfrak v}\equiv n \pmod{q-1}$ and $y_{\mathfrak v}\equiv 0\pmod{q^{d(t+1)}}$.
\end{defn}

\begin{thm}\label{vad7}
Let $s_\mathfrak v\in {\mathbb S}_v(n)$. Then $\hat{f}_{s_{\mathfrak v},n}$ is a $\mathfrak v$-adic modular form
in the sense of Serre.\end{thm}
\begin{proof}
Let $m_i\in {\mathbb S}_{\mathfrak v}(n)$, $i=1,2,\ldots$ be an increasing sequence of positive integer 
converging to $s_v$ in $\mathbb S_{\mathfrak v}$. Put $k_i=m_i+n$ so that $k_i-2n$ is divisible by $(q-1)$
and, for $i\gg 1$, is also positive. Note that if $a\in \mathfrak v$ then $a^{m_i}\to 0$ in $A_\mathfrak v$.
It is now clear that, for $i\gg 1$, the forms $f_{k_i,n}$ of Theorem \ref{Gexpn9} converge
$\mathfrak v$-adically to $\hat{f}_{s_{\mathfrak v},n}$.\end{proof}

\begin{defn}\label{vad7.1}
We call $s_\mathfrak v+n\in {\mathbb S}_v$ the {\it weight} of $\hat{f}_{s_{\mathfrak v},n}$. It's
{\it type} is given by the class of $n$ modulo $q-1$. 
\end{defn}

\begin{rem}\label{vad8}
We do not know if those $\hat{f}_{s_{\mathfrak v},n}$, for $s_\mathfrak v$ not in ${\mathbb S}_{\mathfrak v}(n)$,
are modular forms in the sense of Serre. As of this writing, there is no reason to necessarily believe
that they are.
\end{rem}

\section{Hecke operators}\label{hecke}
We show here how the existence of $\mathfrak v$-adic interpolations of Petrov's forms $f_{k,n}$
has striking implications 
for the action of the Hecke operators on {\it any} fixed form $f_{k_0,n}$. 

Let $g$ be a monic prime of $A$ of degree $d$ and let $\Lambda_g\subset \C_\infty$ be the $\Fq$ submodule of
$g$-division points. Let $G_{k,\Lambda_g}(X)$ be constructed as in section \ref{review}. 
Let $A_d\subset A$ be the $\Fq$ submodule of polynomials of degree $<d$, 
and let $f\in M_{k,m}$ be an arbitrary element. 

\begin{defn}\label{hecke1}
We set
\begin{equation}\label{hecke1.1}
\hat T_gf(z):=\sum_{\beta\in A_d}f\left(\frac{a+\beta}{g}\right)\,,\end{equation} and
\begin{equation}\label{hecke2} 
T_gf(z):=g^kf(gz)+\hat T_gf(z)=g^kf(gz)+\sum_{\beta\in A_d}f\left(\frac{a+\beta}{g}\right)\,.
\end{equation}\end{defn}
One sees, as expected, that $T_g f$ also belongs to $M_{k,m}$ etc. If $f$ has the $u$-expansion
$f=\sum_{n=0}^\infty a_n u^n$, then one has
\begin{equation}\label{hecke3}
T_gf=g^k\sum_{n=0}^\infty a_nu_g^n +\sum_{k=0}^\infty a_k G_{k,\Lambda_g}(gu)\,.
\end{equation}
As a consequence of Equation \ref{hecke3}, we will view $T_g$ and $\hat T_g$ as operators on power series
without referring to the original additive expansion exactly as in classical theory.

\begin{rem}\label{hecke3.1}
It is of fundamental importance that Petrov establishes that, for any $k$ and $n$, the cusp
form $f_{k,n}$ of Theorem \ref{Gexpn9} is a Hecke eigenform for any $g$ with 
associated eigenvalue $g^n$.
\end{rem}

Fix $\mathfrak v$ as before and now also fix $f:=f_{k_0,n}=\sum a_ju^j$ where $k_0$ and $n$ satisfy the hypotheses
of Theorem \ref{Gexpn9}. We decompose $f$ as $f_{0,\mathfrak v}+f_{1,\mathfrak v}=f_0+f_1$ where 
\begin{equation}\label{hecke4} f_0:=\sum_{a\in A_{v,+}} a^{k_0-n} G_n(u_a)\,,\end{equation}
and 
\begin{equation}\label{hecke5} f_1:= \sum_{a\in {\mathfrak v}\cap A_+} a^{k_0-n}G_n(u_a)
\,.\end{equation}
Set $f_i=\sum a^{(i)}_j u^j$ for $i=0,1$. From the definition of $f$ one easily deduces
that $f_1=p_\mathfrak v ^{k_0-n}
\sum a_j u_{p_\mathfrak v}^j$.

\begin{example}\label{hecke5.1}
Let $\mathfrak v=(\theta)$ and $h=\sum_{a\in A_+}a^qu_a$ as before. Thus
$h_{0,\mathfrak v}=\sum_{a\in A_+,\  a(0)\neq 0}a^qu_a$ and $h_{1,\mathfrak v}=\sum_{a\in A_+,\  a(0)=0} a^qu_a$.
\end{example}

Let $g\not\in \mathfrak v$ be a monic irreducible. 
\begin{thm}\label{hecke6}
With the above notation, both $f_0$ and $f_1$ are separately eigenforms for $T_g$ with
eigenvalue $g^n$.
\end{thm}
\begin{proof} Let $k_i$ be an infinite sequence of positive integers converging to $k_0$ in the
$\mathfrak v$-adic weight space ${\mathbb S}_v$. Then, as in the proof of Theorem \ref{vad7}, 
the expansions of $f_{k_i,n}$ converge to that of $f_0$. Note further that {\it each} such cusp form
is an eigenform for $T_g$ with eigenvalue $g^n$. Note further that, by definition, the action of $T_g$ on
$f_{k_i,n}$ converges to the action of $T_g$ on $f_0$. Thus, $T_gf_0=g^nf_0$. Writing 
$f_1=f-f_0$ then gives the statement for $f_1$. \end{proof}

\begin{rems}\label{hecke7}
a. Petrov has pointed out that the techniques of \cite{pe1} allow for a direct proof of 
Theorem \ref{hecke6}.\\
b. Petrov also points out that the decomposition $f=f_0+f_1$ is actually a decomposition of
$\mathfrak v$-adic modular forms. Indeed, one sees that $f_1$ is a modular from for
$\Gamma_0({\mathfrak v})$ and so is also
 a $\mathfrak v$-adic form by Theorem 6.2 of \cite{vi1}.
\end{rems}

Suppose now that $g=p_\mathfrak v$.
\begin{thm}\label{hecke8}
We have 
\begin{equation}\label{hecke9}
\hat T_g f_0=g^n f_0\end{equation}
\end{thm}
\begin{proof} Noting that $g^{k_i}$ tends to $0$ $\mathfrak v$-adically, the result follows as in the
proof of Theorem \ref{hecke6}.\end{proof}
\begin{cor}\label{hecke10}
We have 
\begin{equation}\label{hecke11} 
T_g f_1= g^nf_1-g^{k_0} \sum a_n^{(0)}u_g^n\,.
\end{equation}\end{cor}
\begin{proof}
We have $T_gf=T_gf_0+T_gf_1=g^n f=g^nf_0+g^nf_1$. On the other hand, $T_g f_0=g^{k_0}\sum a_n^{(0)}u_g^n
+\hat T_g f_0$ which, by Theorem \ref{hecke8} equals $g^{k_0}\sum a_n^{(0)} u_g^n+g^n f_0$. The result
follows directly.\end{proof}

One can further decompose $f_0$ and $f_1$ by using a prime $\mathfrak v_1\neq \mathfrak v$
etc.

\end{document}